\documentclass[11pt]{amsart}
\usepackage[paperheight=9.8in,paperwidth=8in,margin=1.15in,heightrounded]{geometry}

\usepackage{amssymb,amsthm,mathrsfs,graphicx,color,tikz}
\newcommand{\ip}[2]{\left\langle {#1}, {#2} \right \rangle}
\newcommand{\abs}[1]{\left\vert {#1}\right\vert}
\newcommand{\norm}[1]{\left\Vert {#1}\right\Vert}
\newcommand{\avg}[1]{\left\langle {#1}\right\rangle}
\newcommand{\ep}{\epsilon}
\newcommand{\vep}{\varepsilon}
\newtheorem{theorem}{Theorem}
\newtheorem{proposition}{Proposition}
\newtheorem{lemma}{Lemma}
\newtheorem{definition}{Definition}

\newcommand{\Apup}{{A_p^{\uparrow}}}

\newcommand{\Atwoup}{{A_2^{\uparrow}}}
\newcommand{\Atwodown}{{A_2^{\downarrow}}}
\DeclareMathOperator{\supp}{supp}
\DeclareMathOperator{\dist}{dist}

\usepackage{amsmath,color,amsrefs,enumitem,enumerate}
\usepackage{hyperref}

\title{An $A_2$ Theorem for One-Sided Calderón-Zygmund Operators}
\author{A. Walton Green}
\address{A. Walton Green, Department of Mathematics, Illinois State University, Campus Box 4520, Normal, IL 61790}
\email{awgree1@ilstu.edu}
\author{Ljupcho Petrov}
\address{Ljupcho Petrov, Department of Mathematics, Washington University in St. Louis, One Brookings Drive, St. Louis, MO 63130}
\email{petrov.l@wustl.edu}
\author{Brett D. Wick}
\address{Brett D. Wick, Department of Mathematics, Washington University in St. Louis, One Brookings Drive, St. Louis, MO 63130}
\email{bwick@wustl.edu}
\thanks{BDW's research is partially supported by National Science Foundation DMS Grant \# 2349868}


\begin{document}

\begin{abstract}
We present a proof of the one-sided $A_2$ theorem in dimension one, with a logarithmic loss. This theorem concerns one-sided Calder\'on-Zygmund operators (CZOs) whose kernels $K(x,y)$ vanish whenever $x < y$. These operators are bounded on $L^2(w)$ provided that the weight $w$ belongs to the one-sided class $A_2^{\uparrow}$. The argument reduces the norm
estimate to testing on indicator functions via a two-weight testing theorem. By combining this with the weak-type $(1,1)$ estimate in the one-sided setting and an extrapolation theorem, we obtain the one-sided $A_2^{\uparrow}$ theorem with a logarithmic loss. We develop a localized theory on fixed intervals by introducing adapted weight classes and showing that the same quantitative bound holds locally for one-sided operators. 
\end{abstract}

\maketitle

\section{Introduction}
One-sided Calderón-Zygmund operators present a fascinating phenomenon in harmonic analysis. By destroying the symmetry of classical singular integrals and truncating their kernels to act only forward or backward, we obtain operators that are bounded on a strictly larger class of weighted spaces. Remarkably, these asymmetric operators satisfy weighted bounds with weight characteristics that can be arbitrarily smaller than what the classical $A_2$ theory requires. This behavior, first observed by Sawyer \cite{sawyer-one-sided} for the one-sided maximal function, suggests that directionality fundamentally changes the weighted theory. Precisely because the admissible weight class is significantly larger (and its characteristic correspondingly smaller), converting this qualitative observation into sharp quantitative bounds has remained a challenging open problem.

While Hytönen's celebrated 2012 resolution of the $A_2$ conjecture \cite{hytonen12} provided optimal bounds for symmetric operators, the one-sided theory has resisted similar progress: neither symmetric dyadic shifts nor sparse domination, the modern tools of weighted theory, can capture the improved bounds that one-sided operators enjoy. This paper provides the first quantitative weighted bound for one-sided Calderón-Zygmund operators.

The classical $A_2$ theorem states that for any Calder\'on-Zygmund operator (CZO) $T$, there exists $C>0$ such that for all weights $w$ in the Muckenhoupt $A_2$ class,
$$
\| T f \|_{L^2(w)} \leq C [w]_{A_2} \| f \|_{L^2(w)}, 
$$
where $[w]_{A_2}$ is the $A_2$ characteristic of the weight $w$. A one-sided Calder\'on--Zygmund operator is a standard CZO whose kernel is truncated so that it vanishes when $y$, the input variable, is larger or smaller than the output variable $x$. This directional truncation fundamentally changes the operator's behavior, but in most cases, by taking advantage of the directional nature, one seeks more refined estimates than those that are available for symmetric singular integrals. Following the terminology of Hytönen--Rosén \cite{hytonen-rosen}, we distinguish between upward- and downward-mapping CZOs.

The central challenge is that one-sided operators destroy the symmetry exploited by all existing proofs of sharp weighted bounds. Symmetric model operators, such as Haar shifts \cite{hytonen12} or sparse operators \cites{lerner-a2-13,lacey-a2-17} do not satisfy one-sided weighted bounds. One-sided analogues of these model operators do satisfy one-sided weighted bounds \cite{lacey-one-sided}, yet representation theorems or sparse domination principles are not yet established in terms of such model operators. The approach we take utilizes the two-weight theory of Nazarov-P\'erez-Treil-Volberg \cites{ntv-jams,ntv-Tb,ntv-lost,ntv-cauchy, ptv}, extracting quantitative bounds through a delicate interplay between testing conditions and Lorentz-space estimates.

Aimar, Forzani, and Mart\'in-Reyes in \cite{aimar-weighted} demonstrated qualitatively that boundedness on $L^2(\mathbb{R}, w)$ of upward-mapping CZOs is characterized by the weight $w$ belonging to the class $A_2^{\uparrow}$. Chill and Kr\'ol \cite{Chill-Krol} extended this result to the case of operator-valued kernels and weights which vanish on a half-line. They were motivated by certain problems in maximal $L^p$ regularity of elliptic PDEs such as \cites{Auscher-Axelsson,Chill-Fiorenza}. We refer to \cite{Chill-Krol}*{Remark 5.2} for a more in-depth discussion of application of weighted and unweighted one-sided singular integral estimates to PDEs. These results did not aim for a sharp quantitative dependence on the operator norm on the weight. Our main result provides the first quantitative bound.

\begin{theorem}\label{thm:A2}
Let $T$ be an upward-mapping CZO. Then for any $w \in A_2^{\uparrow}$,
\begin{equation}\label{e:A2} 
\norm{T}_{L^2(\mathbb{R},w) \to L^2(\mathbb{R},w)} \lesssim [w]_{A_2^{\uparrow}} (1+\log [w]_{A_2^{\uparrow}}).
\end{equation}
\end{theorem}

The $A_2^{\uparrow}$ class generalizes the classical Muckenhoupt $A_2$ class by taking the average of $w$ on the upper half of a cube and of $w^{-1}$ on the lower half (see Definition \ref{d:Ap} below), capturing the directional nature of the operator. The logarithmic factor represents the current gap from the \textit{one-sided $A_2$ conjecture} posed by Chen--Han--Lacey \cite{lacey-one-sided}, which asserts that this factor can be removed. These authors achieved the conjectured bound for the one-sided martingale transform, but the general case remains open.

The main step in proving Theorem \ref{thm:A2} is to apply the two-weight theory to show
\begin{equation}\label{e:T1}  
\norm{T}_{L^2(\mathbb{R},w)\to L^2(\mathbb{R},w)} \lesssim  [w]_{A_2^{\uparrow}} + \sqrt{K_{\mathrm{gl}}},
\end{equation}
where $K_{\mathrm{gl}}$ is the best constant in the global two-weight testing conditions for $T$:
\begin{equation}\label{e:test}
\norm{T(w^{-1} \chi_Q) }_{L^2(\mathbb{R},w)}^2 \le K_{\mathrm{gl}} \int_Q w^{-1}, \qquad\norm{T'(w \chi_Q) }_{L^2(\mathbb{R},w^{-1})}^2 \le K_{\mathrm{gl}} \int_{Q} w. 
\end{equation}
Once \eqref{e:T1} is established, Theorem \ref{thm:A2} follows from showing
\begin{equation}\label{e:KtoLog}
\begin{aligned} 
\sqrt{K_{\mathrm{gl}}} &\lesssim \norm{T}_{L^2(\mathbb{R},w)\rightarrow L^{2,\infty}(\mathbb{R}, w)} + \|T'\|_{L^2(\mathbb{R},w^{-1})\rightarrow L^{2,\infty}(\mathbb{R},w^{-1})} \\
& \lesssim [w]_{A_2^{\uparrow}} (1+\log [w]_{A_2^{\uparrow}}). 
\end{aligned} 
\end{equation}

The first inequality is a Lorentz-space argument, while the second follows by extrapolating the endpoint estimate
\begin{equation}\label{e:A1} 
\norm{T}_{L^1(\mathbb{R},w)\rightarrow L^{1,\infty}(\mathbb{R},w)} \lesssim [w]_{A_1^\uparrow} (1+\log [w]_{A_1^\uparrow}), 
\end{equation}
proved in the standard case by Lerner--Ombrosi--Perez \cite{lerner-a1} and extended to the one-sided, one-dimensional case in \cite{riveros-vidal}.

We prove Theorem \ref{thm:A2} only in the one-dimensional case, as we obtain the testing-constant bound in that setting (Theorems \ref{thm:test} and \ref{thm:main}). In higher dimensions, the situation is more mysterious: it is not even known whether $w \in A_2^{\uparrow}$ qualitatively characterizes boundedness on $L^2(\mathbb{R}^d,w)$ for upward-mapping CZOs. The classical Coifman--Fefferman strategy \cite{coifman74} was carried out in the one-dimensional, one-sided setting in \cite{aimar-weighted}, and the sharp good-$\lambda$ inequality was established in \cite{riveros-vidal}. However, in higher dimensions, these tools remain unavailable, presenting an important challenge for future work.

Beyond the specific bounds obtained, our work demonstrates that the two-weight Nazarov-P\'erez-Treil-Volberg machinery can yield effective quantitative estimates even when dyadic techniques fail. The higher-dimensional case, where even the qualitative theory remains incomplete, highlights how much remains to be understood about operators with inherent asymmetry.

\section{Preliminaries}
We will give the definitions in an arbitrary dimension $d$, though all our results will only hold when $d=1$. 
Let $x=(x_1,\ldots,x_d)$ and $y=(y_1,\ldots,y_d)$ in $\mathbb R^d$. We introduce the following ordering by saying
	\[ x < y, \mbox{ if } x_d < y_d.\]
\begin{definition}\label{def:cz}
A function $K$ in $L^1_{\mathrm{loc}}(\mathbb{R}^{2d} \setminus
\{(x,y) \in \mathbb R^{2d} : x=y\})$ is an upward-mapping Calder\'on-Zygmund kernel if there exist $C,\ep>0$ such that for all $x,y \in \mathbb R^d$, $x\ne y$,
\begin{itemize}
\item[1.] $K(x,y) = 0$ if $x<y$,
\item[2.] $|K(x,y)| \le \frac{C}{|x-y|^d}$,
\item[3.] for all $h \in \mathbb R^d$ such that $x>y$, $x+h>y$, and $|h| \le \tfrac 12 |x-y|$,

\begin{equation}\label{e:smooth}|K(x,y)-K(x+h,y)| + |K(x,y)-K(x,y-h)|<\frac{C|h|^\ep }{|x-y|^{d+\ep }}.\end{equation}
\end{itemize}
We say a linear operator $T:\mathscr S(\mathbb R^d) \to \mathscr S'(\mathbb R^d) $, is an upward-mapping Calder\'on-Zygmund operator (CZO) if
\begin{itemize}
	\item $T$ extends boundedly to $L^2(\mathbb R^d)$; and
	\item there exists an upward-mapping Calder\'on-Zygmund kernel $K$ such that for all $f$ compactly supported and $x \not\in \supp f$,
		\[ Tf(x) = \int K(x,y) f(y) \, dy.\]
\end{itemize}
The dual notion of upward-mapping is downward-mapping. Defining $K'(x,y) = K(y,x)$, we say $K$ is downward-mapping if $K'$ is upward-mapping and $T$ is downward-mapping if $T'$ is upward-mapping. 
\end{definition}

As discussed above, the majority of the literature concerning one-sided CZOs is restricted to dimension one. Two important features arise there. First, in this case, an upward- or downward-mapping CZO is also a standard CZO. One only needs to check that the smoothness condition holds, i.e. \eqref{e:smooth} holds for all $x,y$ and $|h| \le \frac 12 |x-y|$. For such $x,y,h\in \mathbb R$, $x+h$ and $x$ will always lie on the same side of $y$. So if $x<y$, then $x+h<y$ and $x< y-h$; thus $K(x,y) = K(x+h,y) = K(x,y-h) =0$. If $x>y$, then $x+h>y$ and we can apply \eqref{e:smooth}. Therefore, in dimension one, one has more tools available, which are inherited directly from the standard theory ($T1$ theorem, $L^p$ estimates, weak endpoint, $A_p$ weighted estimates). The second important feature in one dimension is that two cubes (intervals) $I$ and $J$ that are disjoint will either satisfy $I<J$ or $J<I$.

The only systematic study of one-sided (there called causal) CZOs in higher dimensions appears to be by Hyt\"onen-Ros\'en \cite{hytonen-rosen}, where they obtain sparse bounds which are sharper than those available for standard CZOs and do provide interesting one-sided estimates which do not hold for standard CZOs. However, they are not sharp enough to obtain weighted estimates in the class $\Atwoup$. In this case, upward or downward kernels are not necessarily standard kernels, yet they do satisfy the H\"ormander smoothness condition \cite{hytonen-rosen}*{Lemma 3.1}. Hence, some of this machinery persists, and in particular, the weak endpoint holds. One could also consider variants of the ordering $x<y$ in higher dimensions (see e.g. \cite{ombrosi-2d}) which may be more or less well-suited to a causal $A_2$ conjecture. 

For a (half-open) cube $I=[a_1,a_1+h) \times \cdots [a_d,a_d+h)$, we denote the lower and upper halves by
	\[ I^- = [a_1,a_1+h) \times \cdots [a_d ,a_d+ \tfrac h2), \qquad I^+ = [a_1,a_1+h) \times \cdots [a_d+ \tfrac h2 ,a_d+h).\]
Concerning the ordering $<$, we can say $I^- < I^{+}$.
We will also use the notation, for $0<q<\infty$, a measure $\mu$, and a measurable set $E \subset \mathbb R^d$,
	\[ \avg{f}_{q,\mu,E} = \left ( \frac{1}{\mu(E)} \int_E \abs{f}^q \, d\mu \right)^{\frac 1q}.\]
When $q=\infty$, this is simply the $L^\infty(E,\mu)$ norm of $f$, and when $q=1$ or $\mu$ is the Lebesgue measure, it is omitted from the notation.

Before introducing the one-sided weight classes, let us explain one technical difference between the one-sided situation and the classical one. Classical $A_p$ weights are assumed to be positive almost everywhere. One justification for this requirement is that CZOs are non-local, in the sense that even if $f$ is compactly supported, $Tf$ vanishes almost nowhere (for some specially chosen $T$). This principle shows that the condition $w>0$ a.e. is necessary for the class of CZOs to be bounded on $L^2(w)$. Upward-mapping CZOs inherit this non-local nature, but only in the upward direction. For this reason, a weight can vanish on some lower half-space
    \[ \mathbb H_z^{+} = \{ x \in \mathbb R^d : x > z\}; \quad \mathbb H^-_z = \{ x \in \mathbb R^d : x < z\}, \quad z \in \mathbb R^d.\]
This observation was made by Sawyer \cite{sawyer-one-sided}, but it has been underemphasized in the subsequent one-sided weighted theory. 
A sublinear operator $T$ is upward-mapping if for each $x \in \mathbb R^d$,
    \[ Tf(x) = 0, \quad \supp f > x. \]
Combining this with the triangle inequality, one obtains
    \begin{equation}\label{e:chi} \chi |Tf| \le \chi |T( \chi f)|,\end{equation}
where $\chi$ is the indicator function of some lower half-space $\mathbb H^-_z$. Therefore, if $w$ is a weight such that an upward-mapping operator $T$ is bounded on $L^p(w)$, then $T$ is also bounded on $L^p(\chi w)$. Indeed,
    \[ \int |Tf|^p \chi w \le \int |T(\chi f)|^p w \lesssim \int |f|^p \chi w.\]
\begin{definition}\label{d:Ap}
A locally integrable function $w$ is called an upward weight if there exists a lower half-space $\mathbb H \in \{\mathbb H^-_z,\mathbb R^d\}$ such that
\begin{itemize}
    \item $w=0$ on $\mathbb R^d \backslash \mathbb H$.
    \item $w$ is positive almost everywhere on $\mathbb H$.
\end{itemize}
Given an upward weight $w$, and $1 \le p <\infty$, define its $\Apup$ characteristic
	\begin{equation}\label{e:Ap} [w]_{\Apup} = \sup_{\substack{I \, \mathrm{cube} \\ I \subset \mathbb H}} \avg{w}_{I^+} \avg{ w^{-1}}_{\frac{1}{p-1},I^{-}} .\end{equation}
Then, $\Apup$ is the collection of all upward weights for which this characteristic is finite. 
\end{definition}

The technical consideration of weights that vanish on a half-space does not greatly affect the proof of the testing theorem, since that is mainly a result about the measures $\mu$ and $\nu$ introduced below, except for an invocation of Sawyer's result \cite{sawyer-one-sided}, which already accounts for vanishing on half-spaces (in this case, half-lines).

The classes of one-sided weights $A_p^{\uparrow}$ and $A_p^{\downarrow}$ are respectively called $A_p^-$ and $A_p^+$ in other papers such as \cite{aimar-weighted}, \cite{martin-reyes-sharp}, \cite{riveros-vidal}, etc.; however, we have chosen to use this notation to make the connection with the one-sided operators explicit: upward-mapping CZOs are bounded on $L^2(w)$ if $w \in \Atwoup$ and downward-mapping CZOs are bounded on $L^2(w)$ if $w \in \Atwodown.$

The classical Muckenhoupt $A_p$ class is defined as above, but with both $I^+$ and $I^-$ replaced by $I$ itself. To see that $\Apup$ is much larger than $A_p$, note that $[w]_{\Apup} \le 2^p [w]_{A_p}$. The same holds for $A_p^{\downarrow}$. Moreover, for any $f>0$ satisfying $f(x) \ge f(y)$ when $x<y$, we have $[fw]_{\Apup} \le [w]_{\Apup}$, and for any such function $f$, we also have $[f]_{\Apup}=1$.

\section{Testing theorem: setup and the easy parts}
In this section, we will outline the setup of the standard two-weight proof, following the outline of \cites{ntv-lost,ptv} and estimate the easy terms in our setting, which hold in all dimensions.

\subsection{Dyadic decomposition}
For a measure $\lambda$ and a dyadic grid $\mathcal D$, let us introduce the disbalanced Haar projections. For $I \in \mathcal D$, let $\mathcal H_I^\lambda$ be the subspace of $L^2(\mathbb R^d, d\lambda)$ consisting of $f$ satisfying
\begin{itemize}
	\item $\supp f \subset I$;
	\item $f$ is constant on each dyadic child of $I$;
	\item $\int f \,d \lambda =0.$
\end{itemize}
It is easy to see that $\mathcal H_I^\lambda \perp \mathcal H_{I'}^\lambda$ for $I \ne I'$. Thus, the orthogonal projections 
	\[ \Delta_I^{\lambda}: L^2(\mathbb R^d, d\lambda) \rightarrow \mathcal H_{I}^\lambda \]
are mutually orthogonal. For a set of cubes $\mathcal R \subset \mathcal D$, define
	\[ \mathbb P^\lambda_{\mathcal R} f = \sum_{I \in \mathcal R} \Delta^\lambda_I f.\]
For a cube $I \in \mathcal D$ and an integer $k$, we will also use the notation
 	\[ \mathcal D(I) = \{ I' \in \mathcal D : I' \subset I \}, \quad \mathcal D_k = \{I \in \mathcal D : \ell(I)=2^k\}.\]

Let us fix a weight $w \in \Atwoup$ and introduce the absolutely continuous measures
	\[ d\mu(x) = \chi_{\mathbb H}(x) w^{-1}(x) \, dx, \qquad d \nu(x) = \chi_{\mathbb H}(x)w(x) \, dx.\]
We will also use the shorthand $T_\mu = T(\chi_{\mathbb H} w^{-1} \cdot)$ and introduce the more precise local testing and weak boundedness constants. Let $K_\chi$ and $K_{\mathrm{WB}}$ be the smallest constants such that for all cubes $Q$
    \begin{equation}\label{e:local-test} \int_Q \abs{T_\mu \chi_Q}^2 \, d\nu \le K_\chi \mu(Q), \end{equation}
and for all $P$ disjoint from $Q$,
    \begin{equation}\label{e:WB} \abs{ \ip{T_\mu \chi_Q}{\chi_P}_{L^2(\mathbb R, d\nu)} } \le K_{\mathrm{WB}} \mu(Q)^{\frac 12} \nu(P)^{\frac 12} .\end{equation}
By standard duality considerations, and noting that $\max \{ K_\chi, K_{\mathrm{WB}} \} \le K_{\mathrm{gl}}$ from \eqref{e:test}, inequality \eqref{e:T1} follows from the next theorem, which is the main technical estimate of our paper.
\begin{theorem}\label{e:twoweight}Let $T$ be an upward-mapping CZO, $w \in \Atwoup$, and $K_\chi$ and $K_{\mathrm{WB}}$ as in \eqref{e:local-test} and \eqref{e:WB}. Then, for any $f \in L^2(\mathbb R,d\mu)$ and $g \in L^2(\mathbb R,d\nu)$,
	\begin{equation}\label{e:form} \abs{\ip{ T_\mu f}{g}_{L^2(\mathbb R,\nu)} }\lesssim \left( [w]_{\Atwoup} + \sqrt{K_\chi} + \sqrt{K_{\mathrm{WB}}} \right) \norm{f}_{L^2(\mathbb R,d\mu)} \norm{g}_{L^2(\mathbb R,d\nu)}, \end{equation}
where the implicit constant depends only on $C,\ep$ from Definition \ref{def:cz}.
\end{theorem}
Note that the following intermediate semilocal testing condition also suffices:
    \[ \int_{2Q} \abs{T_\mu \chi_Q}^2 \, d\nu \le K_{\mathrm{sl}} \mu(Q).\]
Clearly $K_\chi \le K_{\mathrm{sl}} \le K_{\mathrm{gl}}$ while a simple size estimate of the kernel from Definition \ref{def:cz} shows that $K_{\mathrm{WB}} \le C + K_{\mathrm{sl}}$.

Henceforth, we will write $\ip{\cdot}{\cdot}_\nu$, $\norm{\cdot}_\mu$, and $\norm{\cdot}_\nu$ for the above inner product and norms. By the translation and scale invariance of the conditions on $T$ and $w$, we may assume that $f$ and $g$ are supported in a cube, say $[\tfrac 14,\tfrac 34]^d$.  Let $\mathcal{D}^{\mu}$ and $\mathcal{D}^{\nu}$
be two dyadic lattices of
$\mathbb R^d$. Let $\mathcal{D}_0$ be a fixed dyadic lattice containing $[0,1]^d$, and set
 	\[ \mathcal{D}^{\mu} =\mathcal{D}_0+\omega_1, \quad \mathcal{D}^{\nu}=\mathcal{D}_0+\omega_2,\] 
where $\omega_1,\omega_2$ have coordinates in $ [-\frac14,\frac14]$. We have a natural probability space of
pairs of such dyadic lattices, $\{(\omega_1,\omega_2) \in [-\tfrac14,\tfrac14]^{2d}\}$ provided with probability equal to the normalized Lebesgue measure on $[-\frac14,\frac14]^{2d}$. Let $I_0^\mu$ and $I_0^\nu$ be the cubes in $\mathcal D^\mu$ and $\mathcal D^\nu$ respectively which contain $[\tfrac 14,\tfrac 34]^d$, the cube in which $f$ and $g$ are both supported. Then, 
	\[ f = f_1 + f_2, \quad f_1= \chi_{I_0^\mu} \int_{I_0^\mu} f \, d\mu, \qquad g = g_1 + g_2, \quad g_1= \chi_{I_0^\nu} \int_{I_0^\nu} g \, d\nu.\]
The component $f_2$ has mean zero (with respect to $d\mu$, where we assume $\mu \left( {I_0^\mu}\right)=1$) and is supported on $I_0^\mu$. Therefore, $f_1 \perp f_2$ and 
	\[ f_2 = \sum_{\substack{I \in \mathcal D^\mu \\ I \subset I_0^\mu}} \Delta^\mu _I f.\]
One can reduce to establishing \eqref{e:form} for $f_2$ and $g_2$. Indeed, by \eqref{e:WB}, for $j =1,2$,
	\[ \begin{aligned}\abs{\ip{T (w^{-1}f_j)}{g_1}_{L^2(w)} } &\le \abs{\int_{I_0^\nu} g w } \norm{f_j}_{L^2(w^{-1})} \norm{T'(w\chi_{I_0^\nu})}_{L^2(w^{-1})} \\
		&\le \sqrt{K_\chi}\norm{g}_{L^2(w)}\norm{f}_{L^2(w^{-1})};\\
	\abs{\ip{T (w^{-1}f_1)}{g_j}_{L^2(w)} } &\le \abs{\int_{I_0^\mu} f w^{-1} } \norm{g_j}_{L^2(w)} \norm{T(w^{-1}\chi_{I_0^\nu})}_{L^2(w} \\
		&\le \sqrt{K_\chi}\norm{g}_{L^2(w)}\norm{f}_{L^2(w^{-1})}  . \end{aligned}\]
Therefore, we now turn to estimating
	\begin{equation}\label{e:Haar} \sum_{\substack{I \in \mathcal D^\mu(I_0^\mu) \\ J \in \mathcal D^\nu(I_0^\nu)}}\ip{T_\mu \Delta^\nu_I f}{\Delta^\nu_J g}_\nu.\end{equation}
Henceforth, we simply write $\mathcal D^\mu$ and $\mathcal D^\nu$ for $\mathcal D^\mu(I_0^\mu)$ and $\mathcal D^\nu(I_0^\nu)$.
\subsection{Good and bad cubes}
Let $\delta= \frac{\ep}{2(d+\ep)}$. Given a positive integer $r$, a cube $J \in \mathcal D^\nu$ is $r$-bad if there exists $I \in \mathcal D^\mu$ such that $\ell(I) \ge 2^r \ell(J)$ and
	\[ \dist(\partial I,J) \le \ell(J)^\delta \ell(I)^{1-\delta}.\]
We will fix $r$ momentarily and henceforth only refer to bad cubes. If $J$ is not bad, then it is good. The main estimate for good cubes $J$ is that for any $I$ with $\ell(I) \ge 2^r \ell(J)$ and $x \in \partial I$,
	\[ \frac{\ell(J)^{\ep}}{\dist(x,J)^{d+\ep}} \le  \frac{\ell(J)^{\ep}}{\ell(J)^{\delta (d+\ep)} \ell(I)^{(1-\delta)(d+\ep)}}= \frac{\ell(J)^{\ep-\delta (d+\ep)}\ell(I)^{\ep}}{\ell(I)^{d+\ep} \ell(I)^{\ep-\delta (d+\ep)}}\]
Since $\delta$ was chosen so that $\ep-\delta (d+\ep) = \frac{\ep}{2}$, we obtain
	\begin{equation}\label{e:GoodEst} \frac{\ell(J)^{\ep}}{\dist(x,J)^{d+\ep}} \le \left( \frac{\ell(J)}{\ell(I)} \right)^{\frac{\ep}{2}} \ell(I)^{-d}, \qquad x \in \partial I. \end{equation}
It is well-known by now (see e.g. \cite{hytonen12}*{Theorem 3.1} for a precise statement) that if $r$ is a fixed absolute constant large enough, then we can restrict our attention in \eqref{e:Haar} to pairs $(I,J)$ where the smaller cube is always good.

\subsection{Initial splitting}
Let us split the sum in \eqref{e:Haar} into four pieces
	\begin{equation}\label{e:split}
	\sum_{\substack{I \in \mathcal D^\mu\\ J \in \mathcal D^\nu}}\ip{T_\mu \Delta_I^\mu f}{\Delta_J^\nu g}_\nu = \sum_{J < I} + \sum_{I < J} + \sum_{I \sim J} + \sum_{\mathrm{hard}}. \end{equation}
The easiest piece to handle is the first sum, which vanishes due to the support of the kernel. The next easiest piece is the diagonal part, where by $I \sim J$ we mean $2^{-r}\ell(J) \le \ell(I) \le 2^r \ell(J)$ and $\dist(I,J) \le \ell(I) + \ell(J)$. In this case, since $\Delta^\mu_J f$ is the sum of $2^d$ characteristic functions, we obtain
	\[ \abs{\ip{T_\mu \Delta_I^\mu f}{\Delta_J^\nu g}_\nu} \le 2^{2d} \sqrt{K_\chi} \norm{\Delta_I^\mu f}_{\mu}\norm{\Delta_J^\nu g}_{\nu}.\]
Then we can sum using orthogonality and finite overlap of the sets $\{ J : I \sim J\}$. The final easy piece is the second sum, which we split into four regions:
	\[ \begin{aligned}&\mathscr I_1 = \{(I,J) : I < J, \,\dist(I,J) \ge \ell(I) + \ell(J), \, \ell(J) \le \ell(I) \}, \\
		&\mathscr I_2 = \{(I,J) : I < J, \,\dist(I,J) \ge \ell(I) + \ell(J), \, \ell(I) \le \ell(J) \}, \\
		&\mathscr I_3 = \{(I,J) : I < J, \,\dist(I,J) \le \ell(I) + \ell(J), \, 2^r \ell(J) \le \ell(I), \, J \, \mathrm{good} \}, \\
		&\mathscr I_4 = \{(I,J) : I < J, \,\dist(I,J) \le \ell(I) + \ell(J), \, 2^r \ell(I) \le \ell(J), \, I \, \mathrm{good} \}. \end{aligned}\] 
The following standard estimate holds for $(I,J) \in \mathscr I :=\cup_{j=1}^4 \mathscr I_j$:
\begin{equation}\label{e:standard}
	\abs{\ip{T_\mu \Delta_I^\mu f}{\Delta_J^\nu g}_\nu} \lesssim \frac{\ell(I)^{\frac \ep 2}\ell(J)^{\frac \ep 2} \mu(I)^{\frac 12}\nu(J)^{\frac 12} }{(\dist(I,J) 
+\ell(I) +\ell(J))^{d+\ep}}\|\Delta_I^{\mu} f\|_{\mu}\|\Delta_J^{\nu}g\|_{\nu}. \end{equation} 
In all cases, we use the smoothness estimate \eqref{e:smooth} in the standard way, relying on the cancellation of $\Delta^\nu_J g$ (for $\mathscr I_1$ and $\mathscr I_3$) and $\Delta^\mu_I f$ (for $\mathscr I_2$ and $\mathscr I_4$). This step crucially relies on $I<J$ so that $y < x$ for all $y \in I$ and $x \in J$; hence, \eqref{e:smooth} can be applied. In $\mathscr I_1$ and $\mathscr I_2$, this immediately yields a stronger estimate than \eqref{e:standard} with $\ell(I)^{\frac \ep 2} \ell(J)^{\frac \ep 2}$ replaced by $\min\{\ell(J),\ell(I)\}^\ep$, but this is immaterial. In $\mathscr I_3$, we additionally rely on the goodness of $J$ to obtain \eqref{e:GoodEst}, so that for $x \in J$ and $y \in I$,
	\[ \abs{K(x,y) - K(c_J,y)} \lesssim \frac{\ell(J)^\ep}{\dist(y,J)^{d+\ep}} \lesssim \left( \frac{\ell(J)}{\ell(I)} \right)^{\frac{\ep}{2}} \ell(I)^{-d}.\]
This establishes \eqref{e:standard} for $(I,J)$ in $\mathscr I_3$, and $\mathscr I_4$ is handled symmetrically. 

Our next step is to show that, using \eqref{e:standard},
	\begin{equation}\label{e:offdiag} \abs{ \sum_{(I,J) \in \mathscr I} \ip{T_\mu \Delta_I^\mu f}{\Delta_J^\nu g}_\nu } \lesssim [\mu,\nu]_{\Atwoup} \norm{f}_\mu \norm{g}_\nu,\end{equation}
where $[\mu,\nu]_{\Atwoup} = \sup_Q \mu(Q^-) \nu(Q^+)/ |Q|^2 = [w]_{\Atwoup}$.
To do so, we use the following lemma whose proof is postponed to the end of this subsection.
\begin{lemma}\label{l:avg}
For any $\vep>0$, $k, n \in \mathbb Z$, and nonnegative scalars $a_I$,
	\[ \sum_{\ell(J)=2^k} \nu(J) \left( \sum_{\substack{I < J \\ \ell(I)=2^{-n+k}}}  \mu(I)^{\frac 12} \frac{2^{k \vep}}{(\dist(I,J)+2^k)^{d+\vep}} a_I \right)^2 \lesssim [\mu,\nu]_{\Atwoup} \sum a_I^2 .\]
\end{lemma}
We can directly estimate $\mathscr I_1$ and $\mathscr I_3$ by applying Lemma \ref{l:avg} as follows.
	\[ \begin{aligned} &\sum_{(I,J) \in \mathscr I_1 \cup \mathscr I_3} \abs{\ip{T_\mu \Delta_I^\mu f}{\Delta_J^\nu g}_\nu } \lesssim \sum_{n=0}^\infty 2^{-\frac {n\ep}{2}} \sum_{k \in \mathbb Z} \Sigma_{n,k}^{\mathrm{o}}, \\
	&\Sigma_{n,k}^{\mathrm{o}} = \sum_{\ell(J)=2^k} \sum_{\substack{I < J \\ \ell(I)=2^{-n+k}}} \frac{2^{k\ep}}{(\dist(I,J) + 2^k)^{d+\ep}}\mu(I)^{1/2}\nu(J)^{1/2}\|\Delta_I^{\mu}f\|_{\mu}\|\Delta_J^{\nu}g\|_{\nu} . \end{aligned}\]
Cauchy-Schwarz and Lemma \ref{l:avg} show that 
	\[ \abs{\Sigma_{n,k} } \lesssim [\mu,\nu]_{\Atwoup} \norm{\mathbb P_{\mathcal D^\mu_{-n+k} }^\mu f}_\mu \norm{\mathbb P_{\mathcal D^\nu_k}^\nu g}_\nu.\] 
Summing in $n$ and $k$ using Cauchy-Schwarz and mutual orthogonality of the projections $\mathbb P^\lambda_{\mathcal D_j}$ establishes \eqref{e:offdiag} for $(I,J) \in \mathscr I_1 \cup \mathscr I_3$. The other case, $(I,J) \in \mathscr I_2 \cup \mathscr I_4$ is handled in the same way, though one needs Lemma \ref{l:avg} with the condition $I<J$ replaced by $I>J$ and $[\mu,\nu]_{\Atwoup}$ replaced by $[\mu,\nu]_{\Atwodown}$. The proof is the same, except $A_h$ is defined by integrating from $x$ to $x+h$. So we consider \eqref{e:offdiag} proved and conclude the easy part with the proof of Lemma \ref{l:avg}.

\subsubsection{Proof of Lemma \ref{l:avg}}
First, for any $x \in \mathbb R^d$ and $h>0$, let $Q(x,h)$ be the cube centered at $x$ with side length $h$. Then define
	\[ A_hf(x) = \frac{1}{h^d} \int_{Q(x,h)^-} |f| \, d\mu.\]
Assume for now that $\norm{A_h}_{L^2(\mathbb R^d,\mu) \to L^2(\mathbb R^d,\nu)} \le 2^d [\mu,\nu]_{\Atwoup}$. Define
	\[ b_J = \sum_{\substack{I < J \\ \ell(I)=2^{-n+k}}}  \mu(I)^{\frac 12} \frac{2^{k \vep}}{(\dist(I,J)+2^k)^{d+\vep}} a_I, \qquad \phi = \sum_{I} a_I \mu(I)^{-\frac 12} \chi_I.\]
Expand into dyadic annuli $\dist(I,J) \sim 2^j$ for $j \ge k$ so that for any $x \in J$,
	\[ b_J \lesssim 2^{k\vep} \sum_{j=k}^\infty 2^{-j(d+\vep)} \sum_{\substack{I < J \\ \dist(I,J) \sim 2^j \\ \ell(I)=2^{-n+k}}} \mu(I)^{\frac 12} a_I \lesssim 2^{k\vep} \sum_{j=k}^\infty 2^{-j\vep} A_{2^j}\phi(x).\]
Therefore, for each $J$,
	\[ \nu(J)b_J^2 \lesssim \int_J \left( 2^{k\vep} \sum_{j=k}^\infty 2^{-j\vep} A_{2^j}\phi(x) \right)^2 \, d\nu(x).\]
Summing over $J$ (which are disjoint) and using Minkowski's inequality, we obtain
	\[ \left( \sum_{\ell(J)=2^k} \nu(J) b_J^2 \right)^{\frac 12} \lesssim 2^{k\vep} \norm{ \sum_{j=k}^\infty 2^{-j\vep} A_{2^j}\phi }_{\nu} \lesssim [\mu,\nu]_{\Atwoup}^{\frac  12} \norm{\phi}_{\mu}.\] Computing the norm of $\phi$ concludes the proof, yet the norm of $A_h$ must be established. Standard reductions (H\"older's inequality, Fubini's theorem, and rescaling) show that it is enough to establish
	\[ \int_{Q(0,1)^+} \int_{Q(x,1)^-} d \mu(y) d \nu(x) \le 2^d [\mu,\nu]_{\Atwoup}.\]
The best proof (when $d=1$) is the following picture. 

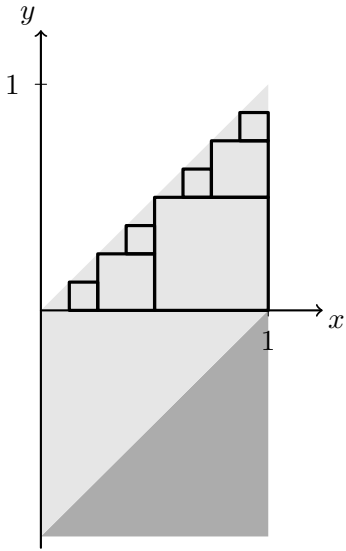
\begin{center}
\begin{tikzpicture}[scale=3, line cap=round, line join=round]
\fill[gray!65] (0,-1) -- (1,0) -- (1,-1) -- cycle;
\fill[gray!20] (0,0) -- (1,1) -- (1,0) -- (0,-1) -- cycle;
\draw[very thick] (0.5,0) rectangle (1,0.5);
\draw[very thick] (0.25,0) rectangle (0.5,0.25);
\draw[very thick] (0.125,0) rectangle (0.25,0.125);
\draw[very thick] (0.25+0.125,0.25) rectangle (0.5,0.25+0.125);
\draw[very thick] (0.75,0.5) rectangle (1,0.75);
\draw[very thick] (0.5+0.125,0.5) rectangle (0.75,0.5+0.125);
\draw[very thick] (0.75+0.125,0.75) rectangle (1,0.75+0.125);
\draw[->,thick] (0,-1.05) -- (0,1.24) node[above left=-2pt] {$y$};
\draw[->,thick] (0,0) -- (1.24,0) node[below right=-2pt] {$x$};
\draw (1,0.025) -- (1,-0.025) node[below=2pt] {$1$};
\draw (0.025,1) -- (-0.025,1) node[left=2pt] {$1$};
\end{tikzpicture}
\end{center}

We are integrating over the light gray region, and we know that for every square $A$ in the picture,
	\begin{equation}\label{e:A} \iint_A d\mu d\nu \le |A| [\mu,\nu]_{\Atwoup}.\end{equation}
Therefore, summing over the squares establishes the claim. In higher dimensions, this same idea works, the key observation again being that we can cover the region of integration $\{(x,y) : x \in Q(0,1)^+, \ y \in Q(x,1)^-\}$ by $Q_j^+ \times Q_j^-$ with $\sum |Q_j|^2 \le 2^d$.
\qed

\section{Hard part in dimension one}
We carry out the remainder of the proof in dimension one. We do this because of the following reason: the hard sum from \eqref{e:split} consists of all pairs $(I,J)$ such that neither $I<J$ nor $J<I$ nor $I\sim J$. In one dimension, as emphasized in the Preliminaries section, if two intervals fail to satisfy $I<J$ and $J<I$, then the intervals must overlap. This forced overlap collapses the geometry of the hard term to a tractable configuration that can be treated by the stopping-time and paraproduct machinery developed below, whereas in higher dimensions, non-comparable cubes can avoid one another entirely. Moreover, our argument depends critically on the sharp weighted bounds for the one-sided maximal operator established in \cites{sawyer-one-sided,martin-reyes-sharp}, which are currently available only in the one-dimensional setting. 

Since we are excluding $I\sim J$, we are left with the following two cases:
	\[ \begin{aligned} &\mathscr I_5 =  \{ (I,J) : I \cap J \ne \varnothing, \ 2^r \ell(J) \le \ell(I), \, J \, \mathrm{good} \}, \\
		&\mathscr I_6 =  \{ (I,J) : I \cap J \ne \varnothing, \ 2^r\ell(I) \le \ell(J) , \, I \, \mathrm{good}\}. \end{aligned} \]
We will only handle $\mathscr I_5$ and $\mathscr I_6$ is treated symmetrically. For $(I,J) \in \mathscr I_5$, since $J$ is good, $J$ must be completely contained in one child of $I$, call this cube $I_{\mathrm{in}}$ and $I_{\mathrm{nb}}= I \setminus I_{\mathrm{in}}$. We first split
	\[ \ip{T_\mu \Delta_I^\mu f}{\Delta_J^\nu g}_\nu = \ip{T_{\mu}(\chi_{I_{\mathrm{nb}}}\Delta_I^{\mu} f)}{ \Delta_J^{\nu}g}_{\nu} +
\ip{T_{\mu}(\chi_{I_{\mathrm{in}}}\Delta_I^{\mu} f)}{\Delta_J^{\nu}g}_{\nu}.\]
The first term is estimated exactly as in the easy section, so the final one remains. $\Delta_I^\mu f$ is in fact constant on $I_{\mathrm{in}}$ and hence it equals its average there. Let $\tilde I$ be a cube in $\mathcal D^\mu$ containing $I$, which will be specified later. We will split
	\begin{equation}\label{e:DS}\ip{T_{\mu}(\chi_{I_{\mathrm{in}}}) \Delta_I^{\mu} f)}{\Delta_J^{\nu}g}_{\nu} = \avg{ \Delta_I^{\mu} f}_{\mu,I_{\mathrm{in}}} \ip{T_{\mu}\chi_{\tilde I}}{\Delta_J^{\nu}g}_{\nu} -\avg{ \Delta_I^{\mu} f}_{\mu,I_{\mathrm{in}}} \ip{T_{\mu}\chi_{\tilde I \setminus I_{\mathrm{in}}}}{\Delta_J^{\nu}g}_{\nu}.\end{equation}

\subsection{Causal pivotal condition}
Given a measure $d\lambda$ and a cube $I$, we introduce the Poisson average of $\lambda$ at $I$,
$$
\mathcal P_I (d \lambda) := \int \frac{\ell(I)^{\ep }}{(\ell(I) +|c(I)-x|)^{d+\ep }} \,d\lambda(x)\,.
$$
Here $c(I)$ denotes the center of the cube $I$, we will often take the measure to be $\chi_{A} \, d\mu$. Given two intervals $J_1=[a,d)$ and $J_2=[b,c)$ with $a<b<c<d$, define the lower and upper components by
\begin{equation}
\label{e:LR}
\mathcal L(J_1 \setminus J_2) = [a,b) , \qquad \mathcal U(J_1 \setminus J_2) = [c,d),
\end{equation}
with the usual understanding that $[a,b)=\varnothing$ if $a=b$.

In the causal pivotal condition below, only the lower portions $\mathcal L(I\setminus I_{\alpha})$ of $I \setminus I_\alpha$ contribute; these are the components of $I \setminus I_\alpha$ that lie below $I_\alpha$.

\begin{lemma}
\label{l:pivotal}
For any cube $I\in \mathcal D^{\mu}$ and any collection of disjoint open cubes $\{I_{\alpha}\}_\alpha$,
\begin{equation}
\label{e:pivotal}
\sum_{\alpha} [\mathcal P_{I_{\alpha}}(\chi_{\mathcal L(I\setminus I_{\alpha})}d\mu)]^2\nu(I_{\alpha}) \lesssim [w]_{\Atwoup}^2 \mu(I).
\end{equation}
\end{lemma}

\begin{proof}
We will use the upward-mapping maximal function
	\[ M^{\uparrow}f(x) = \sup_{h>0} \frac{1}{h} \int_{x-h}^x |f(t)| \,dt .\]
In \cite{sawyer-one-sided, martin-reyes-sharp}, it is shown that 
	\begin{equation}\label{e:sawyer}\norm{M^{\uparrow}}_{L^2(\mathbb R, w) \to L^2(\mathbb R,w)} \lesssim [w]_{\Atwoup}. \end{equation}
Now, it is easy to see that for any $y \in I_{\alpha}$,
$$
\mathcal P_{I_\alpha}(\chi_{\mathcal L(I\setminus I_{\alpha})}d\mu) \lesssim \sum_{k=0}^\infty \frac{2^{-k \ep}}{ [2^{k}\ell(I_{\alpha})]^d} \int_{|x-c(I_\alpha)| \sim 2^k \ell(I_\alpha)} \chi_{\mathcal L(I \setminus I_\alpha)}(x) w^{-1}(x) \, dx.
$$
For each $k$, the integration region can be dominated by $y-h < x < y$ where $h \sim 2^k \ell(I_{\alpha})$, so after summing in $k$,
$$
\mathcal P_{I_{\alpha}}(\chi_{I\setminus I_{\alpha}}d\mu) \le c\,\inf_{y\in I_{\alpha}} M^\uparrow(\chi_{\mathcal L(I\setminus I_{\alpha})} w^{-1})(y)\,.
$$
Next, using the trivial bound $M^\uparrow(\chi_{\mathcal L(I\setminus I_{\alpha})} w^{-1}) \le M^\uparrow(\chi_{I} w^{-1})$ together with the disjointness of $I_{\alpha}$,
$$
\sum_{\alpha} [\mathcal P_{I_{\alpha}}(\chi_{\mathcal L(I\setminus I_{\alpha})}d\mu)]^2\nu(I_{\alpha}) \lesssim\,\int_I \abs{ M^\uparrow (\chi_I w^{-1}) }^2(x) w(x)\, dx.
$$
The desired estimate \eqref{e:pivotal} now follows by \eqref{e:sawyer}.
\end{proof}

\subsection{Sparse collection}
Given a cube $S$, let $\mathcal S(S)$ be the maximal collection of subcubes $Q$ such that
\begin{equation}
\label{e:stop}
\left[ \mathcal P_{Q}(\chi_{\mathcal L(S\setminus Q)}\,d\mu)\right]^2 \nu(Q) \ge 100\,K\, \mu(Q),
\end{equation}
where $K \lesssim [w]_{\Atwoup}$ is the constant in \eqref{e:pivotal}. By Lemma \ref{l:pivotal}, 
	\begin{equation}\label{e:sparse} \sum_{Q \in \mathcal S(S)} \mu(Q) \le \frac 12 \mu(S).\end{equation}
Iterate this process by starting with $\mathcal S_0=\mathcal S(I_0^\mu)$ and for each $k \ge 0$, define
	\[ \mathcal S_{k+1} = \{ S \in \mathcal S(Q): Q \in \mathcal S_k\}.\]
Take now $\mathcal S = \cup_k \mathcal S_k$. For any $I \in \mathcal S$, there exists a unique $k_0$ such that $I \in \mathcal S_{k_0}$. Then, by \eqref{e:sparse}, for each $k>k_0$,
	\[ \sum_{Q \in \mathcal S_k: Q \subset I} \mu(Q) = \sum_{R \in \mathcal S_{k-1}: R \subset I} \sum_{Q \in \mathcal S(R)} \mu(Q) \le \frac 12 \sum_{R \in \mathcal S_{k-1}: R \subset I} \mu(R),\]
and when $k-1=k_0$, the final sum is exactly $\mu(I)$ since $\mathcal S_k$ is always a collection of disjoint cubes.
Iterating this estimate, we obtain
	\begin{equation}\label{e:pack} \sum_{Q \in \mathcal S : Q\subset I} \mu(Q) = \sum_{k=0}^\infty \sum_{Q \in \mathcal S_{k_0+k}} \mu(Q) \le \sum_{k=0}^\infty \left( \frac 12 \right)^k \mu(I) = 2 \mu(I).\end{equation}
To decompose the difficult term, we further introduce the notation for $S \in \mathcal S$, 
	\[ \mathcal O_S = \{I \in \mathcal D^\mu : I \subset S, I \not\subset R \mbox{ for any } R \in \mathcal S(S)\}.\] 
Notice that $\{\mathcal O_S\}_{S \in \mathcal S}$ forms a partition of $\mathcal D^\mu$.
Given $S \in \mathcal S$, let $\hat S$ denote the unique cube in $\mathcal S$ such that $S \in \mathcal S(\hat S)$.
Given $(I,J) \in \mathscr I_5$, define $\tilde I$ in \eqref{e:DS} to be the unique $S \in \mathcal S$ such that $I_i \in \mathcal O_S$. In this way, it remains to estimate the stopping term
	\[ \Sigma_{\mathrm{stop}} = \sum_{(I,J) \in \mathscr I_5} \avg{ \Delta_I^{\mu} f}_{\mu,I_{\mathrm{in}}} \ip{T_{\mu}\chi_{\tilde I \setminus I_{\mathrm{in}}}}{\Delta_J^{\nu}g}_{\nu} \]
and the difficult term
	\[ \Sigma_{\mathrm{diff}} = \sum_{S \in \mathcal S} \sum_{\substack{(I,J) \in \mathscr I_5 \\ I_{\mathrm{in}} \in \mathcal O_S}} \avg{ \Delta_I^{\mu} f}_{\mu,I_{\mathrm{in}}} \ip{T_{\mu}\chi_{S}}{\Delta_J^{\nu}g}_{\nu}.\]
The stopping term is estimated in Proposition \ref{p:stop} below. Let us decompose the difficult term now as
	\[ \Sigma_{\mathrm{diff}} = \mathrm{D}_1 + \mathrm{D}_2 + \mathrm{D}_3,\]
where $\mathrm{D}_1$ consists of all $(I,J) \in \mathscr I_5$ $J$ and $I$ belong to the same $\mathcal O_S$. When we say $J\in \mathcal O_S$, we mean $\mathcal O_S$ defined with $\mathcal D^\nu$ replacing $\mathcal D^\mu$. This is the first paraproduct, $\pi^{(1)}$ below. $\mathrm{D}_2$ consists of all pairs $J \subset I$ with $J$ and $I_{\mathrm{in}}$ belonging to the same $\mathcal O_S$ but $I \in \mathcal O_{\hat S}$. This forces $I_{\mathrm{in}}=S$ and therefore,
	\[ \mathrm{D}_2 = \sum_{S \in \mathcal S} \sum_{J \in \mathcal O_S} ( \Delta^\nu_J g, T_\mu \chi_{S} ) \langle \Delta^\mu_{F(S)}f \rangle_{\mu,S}, \] 
where we are using the notation $F(I)$ to denote the father of $I$ in $\mathcal D^\mu$.
We will use this identity multiple times
	\begin{equation}\label{e:tele} \langle \Delta^\mu_{F(I)}f \rangle_{\mu,I} = \langle f \rangle_{\mu,F(I)} - \langle  f \rangle_{\mu,I}.\end{equation}
The remaining $(I,J)$ contributing to $\mathrm{D}_3$ satisfy $J \in \mathcal O_{S'}$ and $I_{\mathrm{in}} \in \mathcal O_S$ for some $S,S' \in \mathcal S$ such that $S' \subsetneq S$. Furthermore, we can trace the family tree $\mathcal S$ from $S'$ back to $S$ and then all the way back to $I_0^\mu$ by constructing $\{S_j\}_{j=1}^{j_0+1}$ such that
	\[ S'=S_{j_0+1} \subset S_{j_0} \subset \cdots \subset S_2\subset S_1 = I_0^\mu, \quad S_{j_1}=S, \quad S_{j} \in \mathcal S(S_j).\]
Exchanging the sums and relabelling $I_{\mathrm{in}}=I$, we have
	\[ \mathrm{D}_3 = \sum_{S' \in \mathcal S} \sum_{j=1}^{j_0} \sum_{J \in \mathcal O_{S'}} \langle \Delta^\nu_J g,T_\mu \chi_{S_j} \rangle \sum_{I \in \mathcal O_{S_j}, J \subset I} \langle \Delta^\mu_{F(I)} f \rangle_{I}.\]
By telescoping \eqref{e:tele}, for each $j$, the last term is 
	\[ \sum_{I \in \mathcal O_{S_j}, J \subset I} \langle \Delta^\mu_{F(I)} f \rangle_{I} = \langle f \rangle_{\mu,F(S_{j})}-\langle f \rangle_{\mu,F(S_{j+1})}.\]
Then, since $S_{j+1} \subset S_{j}$,
	\[ \sum_{j=1}^{j_0} \chi_{S_j}(\langle f \rangle_{\mu,F(S_{j})}-\langle f \rangle_{\mu,F(S_{j+1})}) = \chi_{S_{j_0}}(\langle f \rangle_{\mu,F(S_{1})}-\langle f \rangle_{\mu,F(S_{j_0+1})}) \]
	\[ +  \sum_{j=1}^{j_0-1}\sum_{k=j}^{j_0-1} \chi_{S_k \setminus S_{k+1}}(\langle f \rangle_{\mu,F(S_{j})}-\langle f \rangle_{\mu,F(S_{j+1})})  \]
Exchanging the sums, and recalling that $\langle f \rangle_{\mu,S_{1}}=0$, the final term equals
	\[ \sum_{k=1}^{j_0-1} \chi_{S_k \setminus S_{k+1}} \sum_{j=1}^{k} (\langle f \rangle_{\mu,F(S_{j})}-\langle f \rangle_{\mu,F(S_{j+1}})) = -\sum_{k=1}^{j_0-1} \chi_{S_k \setminus S_{k+1}}\langle f \rangle_{\mu,F(S_{k+1})}.\]
Therefore,
$\mathrm{D}_3= -\mathrm{D}_4 - \mathrm{D}_5$, where after relabeling and reorganizing,
	\[ \begin{aligned}&\mathrm{D}_4 = \sum_{S \in \mathcal S} \langle \mathbb P^\nu_{\mathcal O_S} g, T_\mu \chi_{\hat {S}} \rangle_\nu \langle f \rangle_{\mu,{F(S)}}, \\
	&\mathrm{D}_5 = \sum_{S \in \mathcal S} \sum_{Q \in \mathcal S(S)} \langle \mathbb P^\nu_{\mathcal D(Q)} g, T_\mu \chi_{\hat S \setminus S} \rangle_\nu \langle f \rangle_{\mu,{F(S)}}. \end{aligned}\]
Since $\mathcal D^\nu(S) = \mathcal O^\nu_S \cup (\cup_{Q  \in \mathcal S(S)} \mathcal D^\nu(Q) )$, splitting $\chi_{\hat S} = \chi_S + \chi_{\hat S \setminus S}$, we further rearrange so that $\mathrm{D}_4 + \mathrm{D}_5=\mathrm{D}_6 + \mathrm{D}_7$, where
	\[  \mathrm{D}_6 = \sum_{S \in \mathcal S} \langle \mathbb P^\nu_{\mathcal O_S} g, T_\mu \chi_{S} \rangle_\nu \langle f \rangle_{\mu,{F(S)}}, \quad \mathrm{D}_7= \sum_{S \in \mathcal S} \langle \mathbb P^\nu_{\mathcal D(S)} g, T_\mu \chi_{\hat S \setminus S} \rangle_\nu \langle f \rangle_{\mu,{F(S)}}.\]
Finally, 
	\[ \mathrm{D_2} - \mathrm{D}_6 = -\sum_{S \in \mathcal S} \langle \mathbb P^\nu_{\mathcal O_S} g, T_\mu \chi_{\hat {S}} \rangle_{\nu} \langle f \rangle_{\mu,S}, \]
and also,
	\[ \Sigma_{\mathrm{diff}} = \mathrm{D}_1 + (\mathrm{D_2} - \mathrm{D_6}) - \mathrm{D_7} =: \sum_{i=1}^3 \ip{\pi^{(i)}f}{g}_{\nu},\]
and these are the three paraproduct forms. 

\subsection{Proofs of easy paraproduct bounds}
Here we give the estimates for $\Sigma_{\mathrm{stop}}$ and the first two paraproducts. The simplest is $\pi^{(2)}f$, which is controlled by conditions \eqref{e:pack} and $K_\chi$. Indeed, by the orthogonality of $\mathbb P^\nu_{\mathcal O_S}$ and the fact that $\mathbb P^\nu_{\mathcal O_S} g$ is supported in $S$, 
	\[ \Vert \pi^{(2)}f \Vert_{L^2(d\nu)}^2 \lesssim \sum_{S \in \mathcal S} b_S \langle f \rangle_{\mu,{S}}^2  , \qquad b_S = \left\Vert T_\mu \chi_{S} \right\Vert_{L^2(S,d\nu)}^2.\]
By \eqref{e:local-test}, $b_S \le K_{\chi} \mu(S)$ and thus the Carleson condition follows from \eqref{e:pack}. 
To estimate $\Sigma_{\mathrm{stop}}$ and $\pi^{(1)}$, we connect to the pivotal condition through the following lemma.
\begin{lemma}\label{l:poisson} Let $(I,J) \in \mathscr I_5$ and $S \in \mathcal D^\mu$ such that $I\subset S$. Then,
	\[ \abs{\ip{ T_{\mu}\chi_{S\setminus I})}{\Delta_J^{\nu}g}_{\nu}} \lesssim \nu(J)^{\frac 12}\norm{\Delta_J^{\nu}g}_{\nu}\left (\frac{\ell(J)}{\ell(I)}\right)^{\ep/2} \mathcal P_{I_{\mathrm{in}}}(\chi_{\mathcal L(S\setminus I)} \,d\mu).\]
\end{lemma}
\begin{proof}
Partition $S \setminus I = \mathcal L(S \setminus I) \cup \mathcal U(S \setminus I)$. For $x \in J \subset I$, and $y \in \mathcal U(S \setminus I)$, $K(x,y)=0$. Furthermore, if $y \in \mathcal L(S \setminus I)$, then $y<x$ and $y<c(J)$ therefore, we may apply \eqref{e:smooth} together with the cancellation of $\Delta^\nu_J g$ to obtain
\[ \abs{\ip{T_{\mu}\chi_{S\setminus I}}{\Delta_J^{\nu}g}_{\nu}} \lesssim \left( \int_{\mathcal L(S\setminus I)}\frac{\ell(J)^{\ep}}{\dist(y,J)^{1+\ep}}\, d\mu(y) \right) \norm{\Delta_J^{\nu}g}_{L^1(\nu)}.\]
Since $J$ is good, for $y \not\in I$, by \eqref{e:GoodEst},
	\[ \frac{\ell(J)^{\ep}}{\dist(y,J)^{1+\ep}} \lesssim \left( \frac{\ell(J)}{\ell(I)} \right)^{\ep/2} \frac{\ell(I)^\ep}{(\ell(I) + |y-c(I_{\mathrm{in}})|)^{d +\ep }},\]
which concludes the proof after Cauchy-Schwarz.
\end{proof}
\begin{proposition}\label{p:stop} The following estimate holds:
	\[ \abs{ \Sigma_{\mathrm{stop}} } \lesssim K \norm{f}_\mu \norm{g}_{\nu}.\]
\end{proposition}
\begin{proof}
The proof is similar to \eqref{e:offdiag}. Apply Lemma \ref{l:poisson} and split
	\[ \Sigma_{\mathrm{stop}} \lesssim \sum_{\pm} \sum_{n=0}^\infty 2^{-\frac{n\ep}{2}} \sum_{k \in \mathbb Z} \Sigma_{n,k}^{\pm},\]
where
	\[ \Sigma_{n,k}^{\pm} = \sum_{\substack{ (I,J) \in \mathscr I_5\\ \ell(I)= 2^{k}\\ \ell(J)= 2^{-n+k} \\ J \subset I^{\pm}}} \left(\frac{\nu(J)}{\mu(I^{\pm})}\right)^{\frac 12} \mathcal P_{I^{\pm}}(\chi_{\mathcal L(\tilde{I}\setminus I^{\pm})} \,d\mu)\|\Delta_J^{\nu}g\|_{\nu}\|\Delta_I^{\mu}f\|_{\mu}.\]
Applying Cauchy-Schwarz to the sum in $J$, and using orthogonality in the same way as in the proof of \eqref{e:offdiag}, it remains to estimate, for each $n$, $k$, and $I$,
	\[  \sum_{\substack{J \subset I^{\pm} \\ \ell(J)=2^{-n+k} }} \frac{\nu(J)}{\mu(I^{\pm})}\mathcal P_{I^{\pm}}(\chi_{\mathcal L(\hat{I}\setminus I^{\pm})}\,d\mu)^2.\]
However, recall that $\tilde I$ was chosen so that the child of $I$ containing $J$, in this case $I^{\pm}$, belongs to $\mathcal O_{\tilde I}$. In other words, $I^{\pm}$ is a \textit{non-stopping cube} and therefore \eqref{e:stop} does not hold for $S=\tilde I$ and $Q=I^{\pm}$. In fact, the opposite holds:
	\[ \mathcal P_{I^{\pm}}(\chi_{\mathcal L(\hat{I}\setminus I^{\pm})}\,d\mu)^2 \nu(I^{\pm}) \lesssim K \mu(I^{\pm}).\]
Combining the above two displays concludes the proof.\end{proof}
\begin{proposition}\label{p:pi1} The first paraproduct is bounded as follows:
	\[ \norm{\pi^{(1)}f}_{L^2(\nu)} \lesssim \sqrt{K} + \sqrt{K_\chi}.\]
\end{proposition}
\begin{proof}
Recall 
	\[ \ip{\pi^{(1)} f}{g}_\nu = \sum_{S \in \mathcal S} \sum_{(I,J) \in \mathscr I_5 \cap \mathcal O_S} \avg{ \Delta_I^{\mu} f}_{\mu,I_{\mathrm{in}}} \ip{T_{\mu}\chi_{S}}{\Delta_J^{\nu}g}_{\nu}.\]
Let us fix $J$ and sum over $I \supset J$ with $\ell(I) \ge 2^r \ell(J)$ and $I \subset S$. By \eqref{e:tele}, 
	\[ \sum_{I}\avg{ \Delta_I^{\mu} f}_{\mu,I_{\mathrm{in}}} = \sum_I \avg{ \Delta_I^{\mu} \mathbb P^\mu_{\mathcal O_S} f}_{\mu,I_{\mathrm{in}}} = \avg{\mathbb P^\mu_{\mathcal O_S} f}_{\mu,S} - \avg{\mathbb P^\mu_{\mathcal O_S} f}_{\mu,L(J)}, \]
where $L(J) \in \mathcal D^\mu$ satisfies $\ell(L(J))\sim 2^r \ell(J)$ and $J \subset L(J)$. Furthermore, the average over $S$ vanishes since $\mathbb P^\mu_{\mathcal O_S} f$ is supported inside $S$ and has mean zero. Therefore,
	\[  \ip{\pi^{(1)} f}{g}_\nu = \sum_{S \in \mathcal S} \sum_{J \in \mathcal O_S} \avg{ \mathbb P^\mu_{\mathcal O_S} f}_{\mu,L(J)} \ip{\Delta_J^{\nu}T_{\mu}\chi_{S }}{\mathbb P^\nu_{\mathcal O_S}g}_{\nu}.\]
By orthogonality of $\mathbb P_{\mathcal O_S}$, it remains to estimate for each $S$, 
	\[ \norm{\pi^{(1)}_S f}_\nu^2 = \sum_{I \in \mathcal O_S}  \avg{f}_{\mu,I}^2 a_I, \quad a_I = \sum_{\substack{J \in \mathcal O_S\\ \ell(J)\sim 2^{-r} \ell(I)}} \norm{\Delta_J^{\nu}T_{\mu}\chi_{S }}_\nu,\]
where we have relabelled $L(J)$ as $I$. To establish the Carleson condition on $a_I$, rewrite
	\[ \sum_{\substack{L \in \mathcal O_S\\ L \subset I}} a_L = \sum_{\substack{J \in \mathcal O_S \\ (I,J) \in \mathscr I_5}}  \norm{\Delta_J^{\nu}T_{\mu}\chi_{S }}_\nu^2.\]
By orthogonality of $\Delta^\nu_J$, duality, and the testing condition \eqref{e:local-test} on $I$, it remains to estimate, by the same argument as in the proof of Proposition \ref{p:stop}, crucially relying on $I \in \mathcal O_S$ in the same way,
	\[ \sum_{\substack{J \in \mathcal O_S \\ (I,J) \in \mathscr I_5}}  \abs{ \ip{T_{\mu}\chi_{S\setminus I }}{\Delta_J^{\nu} \phi}_\nu } \lesssim K \norm{\phi}_\nu. \] 
\end{proof}

\subsection{The third paraproduct}
For two cubes $S' \subset S$ in the sparse collection $\mathcal S$, there is a unique chain $\{S_0,\ldots,S_N\} \subset \mathcal S$ such that 
	\[ S_0=S', \quad S_N=S, \quad S_{j+1} = \hat S_j, \qquad j=0,\ldots,N-1.\]
Define $r(S',S)=N$, which is the tree distance in $\mathcal S$. Define for each $S \in \mathcal S$, and $j=0,1,2,\ldots$
\[ \begin{aligned}
a^j_S :=\sum_{S'\in\mathcal S, S'\subset S, r(S', S) =j} \norm{ \mathbb{P}_{\nu, \mathcal D(S')} T_{\mu}\chi_{\hat{S}\setminus S}}_{\nu}^2,\end{aligned} \]
\begin{proposition}\label{e:pi3}
Suppose for each $I \in \mathcal S$,
	\begin{equation}\label{e:carla} \sum_{\substack{S \in \mathcal S \\ F(S) \subset I}} a^j_S \le A_j \mu(I).\end{equation}
Then, for any $0<b<1$
	\[ \pi^{(3)} \lesssim \frac{1}{1-b}A_0 + \sum_{j=1}^\infty b^{-j} A_j.\]
\end{proposition}
\begin{proof}
Expanding 
	\[ \norm{ \pi^{(3)} f }_{L^2(d\nu)}^2 = \sum_{S,S' \in \mathcal S} \avg{ f}_{\mu,F(S')}\avg{ f}_{\mu,F(S)} \ip{ \mathbb{P}^\nu_{\mathcal D^\nu(S)} T_{\mu} \chi_{\hat{S}\setminus S}}{ \mathbb{P}^\nu_{\mathcal D^\nu(S')} T_{\mu}\chi_{\hat{S'}\setminus S'} }_{\nu} \]
The composition of the projections $\mathbb{P}^\nu_{\mathcal D^\nu(S')}\mathbb{P}^\nu_{\mathcal D^\nu(S)}$ vanishes whenever $S$ and $S'$ are disjoint and returns the smaller cube when one is contained in the other. Therefore, we may estimate for any $0<b<1$ and Cauchy-Schwarz,
	\[ \begin{aligned} \norm{\pi^{(3)} f}_{L^2(\nu)}^2 &\lesssim \sum_{j=0}^\infty \sum_{S \in \mathcal S} \sum_{\substack{S' \in \mathcal S \\ S' \subset S \\ r(S',S)=j}} \avg{ f}_{\mu,F(S)}^2 b^{-j} \norm{ \mathbb{P}^\nu_{\mathcal D^\nu(S')} T_{\mu} \chi_{\hat{S}\setminus S}}_\nu^2 \\
	&+\sum_{j=0}^\infty \sum_{S \in \mathcal S} \sum_{\substack{S' \in \mathcal S \\ S' \subset S \\ r(S',S)=j}} \norm{ \mathbb{P}^\nu_{\mathcal D^\nu(S')} T_{\mu}\chi_{\hat{S'}\setminus S'} }_{\nu}^2 b^{j} \avg{ f}_{\mu,F(S')}^2. \end{aligned} \]
The first triple sum above is controlled by $\sum_{j=0}^\infty b^{-j} A_j$, where $A_j$ satisfies \eqref{e:carla}. We claim that the second sum is controlled by $\frac{1}{1-b} A_0$. Indeed, exchanging the sums, we need to estimate
	\[ \sum_{S' \in \mathcal S} \sum_{j=0}^\infty b^j \sum_{\substack{S \supset S' \\ r(S',S)=j}} a_{S'}^0 \avg{ f}_{\mu,F(S')}^2.\]
However, the cardinality of the innermost sum is one; therefore, we sum in $j$ and appeal to \eqref{e:carla} with $j=0$.
\end{proof}
It remains to verify the Carleson conditions \eqref{e:carla}.
\begin{proposition}\label{p:carla} The estimate \eqref{e:carla} holds with $A_j \lesssim 2^{-j\ep}K$.
\end{proposition}
A key ingredient will be the following result.
\begin{lemma}
\label{l:carlaP}
Let $S',S,I \in \mathcal S$ with $S' \subset S \subset I$ and $r(S',S)=j$. Then,
\[ \norm{ \mathbb P_{\nu, \mathcal D(S')} (T_{\mu}\chi_{I\setminus S}) }_{\nu}^2 \lesssim 2^{-j\ep} \nu(S')(\mathcal P_S\chi_{\mathcal L(I\setminus S)} d\mu)^2. \]
\end{lemma}

\begin{proof}
By duality and Lemma \ref{l:poisson}, 
\[ \norm{ \mathbb{P}_{\nu, \mathcal D^\nu(S')}( T_{\mu}\chi_{I\setminus S})}_{\nu}^2\le \left[ \mathcal P_S(\chi_{\mathcal L(I\setminus S)} \,d\mu)\right]^2 \sum_{\substack{J \in \mathcal D^\nu \\  
J\subset S'}} \nu(J)\left(\frac{\ell(J)}{\ell(S)}\right)^{\ep}. \]
However, since $r(S',S)=j$, $\ell(S') \le 2^{-j} \ell(S)$, and therefore the summation in the above display can be estimated as
	\[ \sum_{\substack{J \in \mathcal D^\nu \\ J\subset S'}} \nu(J)\left(\frac{\ell(J)}{\ell(S)}\right)^{\ep} = \sum_{i=j}^\infty 2^{-i \ep} \sum_{\substack{J \in \mathcal D^\nu \\  
J\subset S' \\ \ell(J)=2^{-i} \ell(S) }} \nu(J) \lesssim 2^{-j\ep}\nu(S').\]
\end{proof}

\begin{proof}[Proof of Proposition \ref{p:carla}]
Let $I \in \mathcal S$ and let $\mathcal F(I)$ be the maximal subcubes $S \in \mathcal S$ such that $F(S)\subset I$. Note that if $S \in \mathcal F(I)$, then $\hat S = I$. If this failed, then there would be $S' \in \mathcal S$ such that $S \subsetneq S' \subsetneq I$ and in particular $F(S') \subset I$, contradicting the maximality of $S$. Let $\{F_{\beta}\}_{\beta\in B}$ denote the family of maximal cubes among $\{F(S)\}_{S \in \mathcal F(I)}$. For each $\beta \in B$, let $\{S_{\beta,\gamma}\}_\gamma$ denote all cubes from $\mathcal F(I)$ contained in $F_{\beta}$. Then, by Lemma \ref{l:carlaP},
\begin{equation}\label{e:ppack}\begin{aligned}
\sum_{S \in \mathcal F(I)} &a_S^j = \sum_{\beta\in B} \sum_{\gamma} \sum_{\substack{S' \subset S_{\beta,\gamma} \\ r(S',S)=j}}\norm{ \mathbb P^\nu_{\mathcal D^\nu(S')} (T_{\mu}\chi_{I \setminus S_{\beta,\gamma}})}_{\nu}^2  \\
& \lesssim \sum_{\beta\in B} \sum_{\gamma}  \sum_{\substack{S' \subset S_{\beta,\gamma} \\ r(S',S)=j}} \|\mathbb P^\nu_{\mathcal D^\nu(S')} (T_{\mu}\chi_{I\setminus F_{\beta}})\|^2_{\nu} + \|\mathbb P^\nu_{\mathcal D^\nu(S')} (T_{\mu}\chi_{F_{\beta}\setminus S_{\beta,\gamma} })\|^2_{\nu} \\
&\lesssim 2^{-j\ep} \sum_{\beta,\gamma,S'} \nu(S') \left[ \mathcal P_{F_\beta}(\chi_{\mathcal L(I \setminus F_\beta)} \, d\mu)^2 + \mathcal P_{S_{\beta,\gamma}}(\chi_{\mathcal L(F_\beta \setminus S_{\beta,\gamma})} \, d\mu)^2 \right] \\
&\lesssim 2^{-j\ep} K \mu(I),\end{aligned} \end{equation}
where the final inequality follows by the pairwise disjointness of each of the collections $\{S' \subset S : r(S',S)=j\}$, $\{S_{\beta,\gamma}\}_\gamma$ and $\{F_\beta\}_\beta$ together with Lemma \ref{l:pivotal}. Let $\mathcal F_0=\mathcal F(I)$ and for each $k\ge 0$ define
	\[ \mathcal F_{k+1} = \{ R \in \mathcal F(Q): Q \in \mathcal F_{k} \}.\]
Then, for each $k \ge 1$, by \eqref{e:ppack},
	\[ \sum_{S \in \mathcal F_k}a_S^j = \sum_{Q \in \mathcal F_{k-1}} \sum_{S \in \mathcal F(Q)} a_S^j \lesssim 2^{-j\ep} K \sum_{Q \in \mathcal F_{k-1}} \mu(Q).\]
Therefore, summing in $k$ and applying \eqref{e:sparse} establishes \eqref{e:carla}.
\end{proof}

\section{The One-Sided $A_2$ theorem with logarithmic loss}
We need to control the testing conditions. We will control both $K_\chi$ and $K_{\mathrm{WB}}$ by the global testing conditions:
    \[ \norm{T(w^{-1} \chi_I)}_{L^2(w)}^2 \le K_{\mathrm{gl}} w(I), \quad \norm{T'(w \chi_I)}_{L^2(w^{-1})}^2 \le K_{\mathrm{gl}} w^{-1}(I), \quad \forall I \subset \mathbb H. \]
    
\begin{theorem}\label{thm:test} The global testing constant satisfies:
    \[ \sqrt{K_{\mathrm{gl}}} \lesssim [w]_{A_2^\uparrow} \log\left( e + [w]_{A_2^\uparrow} \right).\]
\end{theorem}

The main step in proving Theorem \ref{thm:test} is a weak-type estimate. For $w$ which are non-vanishing (i.e. $\mathbb H=\mathbb R$), the result we will use (Proposition \ref{prop:weak} below) was proved in \cite{riveros-vidal}. To extend it to the case where $\mathbb H \ne \mathbb R$, we will prove a local weak-type (1,1) estimate (Theorem \ref{thm:main}) in Section \ref{s:endpoint}. From that result, we obtain the following by extrapolation.
\begin{proposition}\label{prop:weak} Let $T$ be an upward mapping CZO. There exists $C>0$ such that for all $w \in \Atwoup$
    \[ \norm{T}_{L^2(w) \to L^{2,\infty}(w)} \le C [w]_{A_2^\uparrow} \log\left( e + [w]_{A_2^\uparrow} \right).\]
\end{proposition}
The proof of Proposition \ref{prop:weak} is postponed until Section \ref{ss:extrap}. Equipped with Proposition \ref{prop:weak} (and the downward-mapping analogue by duality), Theorem \ref{thm:test} now follows from a standard Lorentz-space argument, namely that restricted strong-type estimates are dual to weak-type; see \cite{ptv}*{Theorem 2.6}. 

Proposition \ref{prop:weak} has been proved in the classical case by Lerner-Ombrosi-Perez in \cite{lerner-a1}. It was proven by Riveros-Vidal in \cite{riveros-vidal} for one-sided CZOs when $\mathbb H = \mathbb R$. In fact, they prove the $L^1$ weighted weak-type estimate. Therefore, we prove Theorem \ref{thm:main} from which Proposition \ref{prop:weak} follows by extrapolation. The extrapolation in the one-sided setting uses the sharp bounds for the one-sided maximal function and the usual Rubio-de-Francia construction \cite{Cuerva_Francia_Weighted_Norm}: \[
R_{}^{\uparrow}h \;=\;
\sum_{k\ge 0}
\bigl(2\,\|M_{}^{\uparrow}\|_{L^{2}(w)}\bigr)^{-k}\,
\bigl(M_{}^{\uparrow}\bigr)^{k}h.
\]

Theorem \ref{thm:A2} is established only in dimension one, since the testing-constant estimate has been proved only in that dimension. In addition, we restrict to dimension one because the proof relies fundamentally on the available results for the one-sided maximal operator.

\section{Localized Weighted weak-type estimate}\label{s:endpoint}

Let $I_0=(a_0,b_0)$ be a fixed interval in $\mathbb{R}$ with $-\infty \le a_0 < b_0 \le \infty$.
For a function $f$ on $I_0$, the localized one-sided (right and left) Hardy-Littlewood maximal operators are defined as
\[
M_{I_0}^{\downarrow}f(x)=\sup_{\substack{h>0\\x+h\leq b_0}}\frac{1}{h}\int_{x}^{x+h}|f(t)|\,dt,\quad x\in I_0,
\]
\[
M_{I_0}^{\uparrow}f(x)=\sup_{\substack{h>0\\x-h\geq a_0}}\frac{1}{h}\int_{x-h}^{x}|f(t)|\,dt,\quad x\in I_0.
\]

Therefore, the suprema are respectively taken over $h>0$ such that the intervals $(x,x+h)$ and $(x-h,x)$ are subsets of $I_0$. For $r \ge 1$, we also denote 
$$M^{\uparrow / \downarrow}_{I_0,r} f = \left(M^{\uparrow / \downarrow}_{I_0} |f|^r \right)^{\frac 1r}.$$

For $x \in I_0$, we have
    \begin{equation}\label{e:MI0} M^{\uparrow}_{I_0,r}f(x) = M^{\uparrow}_{\mathbb R,r} (\chi_{I_0}f)(x).\end{equation}
Hence, many facts about the maximal function immediately follow from the standard case ($I_0=\mathbb R$). In particular, we obtain Lemma \ref{l:Mr-_bound} below. However, the main difference is that in the weight class here, the weight is only assumed to be non-vanishing on $I_0$; in addition, the $A_p^{\uparrow}$ conditions are only imposed on $I \subset I_0$. More precisely, for $p > 1$, we say a positive a.e. function $w$ on $I_0$ belongs to $A_p^{\uparrow}(I_0)$ if there exists $C>0$ such that for all $a<c<x<d<b$ there holds
    \begin{equation}\label{eq:Ap-local} \int_x^d w(y)\, dy \left( \int_c^x w(y)^{\frac{-1}{p-1}} \, dy \right)^{p-1} \le C(d-c)^{p}.\end{equation}
In general, this condition is stronger than simply imposing \eqref{e:Ap} on subintervals of $I_0$. Denote by $[w]_{A_p^{\uparrow}(I_0)}$ the smallest $C$ such that \eqref{eq:Ap-local} holds.

Taking the limit of \eqref{eq:Ap-local} as $p \to 1$ we obtain
    \begin{equation}\label{e:A1-local-2}\frac{1}{d-c}\int_x^d w(y) \, dy \le C \inf_{c<y<x} w(y).\end{equation}
This implies, by taking $c \to x$, that
    \begin{equation}\label{e:A1-local} M^{\downarrow}_{I_0}w(x) \le Cw(x)\end{equation}
for all $x \in I_0$. On the other hand, for any $c<y<x$,
    \[ \frac{1}{d-c} \int_x^d w(t) \, dt \le \frac{1}{d-y} \int_y^d w(t) \, dt \le M^{\downarrow}_{I_0}(y) \le C w(y).\]
Therefore, we define $[w]_{A_1^{\uparrow}(I_0)}$ to be the infimum over all $C$ such that \eqref{e:A1-local} or equivalently \eqref{e:A1-local-2} holds.

\subsection{Preliminaries to the Proof of Theorem \ref{thm:main}}
For our proof of Theorem \ref{thm:main}, we will use the following three facts. The first is the standard weighted endpoint result for the maximal function, which we prove in Section \ref{ss:max} below. 
\begin{proposition}\label{thm:weak-type} 
If $w$ belongs to $A^{\uparrow}_1(I_0)$, then 
    \[ \norm{M^{\uparrow}_{I_0}}_{L^1(w) \to L^{1,\infty}(w)} \lesssim [w]_{A^{\uparrow}_1(I_0)}.\]
\end{proposition}  

The second ingredient we require is the sharp reverse H\"older inequality. The general global inequality is given in \cite{riveros-vidal}*{Lemma 3.1}. For the case of localized $A^{\uparrow}_1(I_0)$ weights, we present a shorter proof.
\begin{lemma}\label{l:wRH} There exists an absolute constant $C>0$ such that for any $w\in A_{1}^{\uparrow}(I_0)$,
 there holds
\begin{equation}\label{e:wRH}
\int_{a}^{b}w^{r}\leq 2\left(M_{I_0}^{\downarrow}(w\chi_{(a,b)})(b)\right)^{r-1}\int_{a}^{b}w, \quad r = 1 + \frac{1}{C[ w ]_{A_{1}^{\uparrow}(I_0)}},
\end{equation}
for every bounded interval $(a,b)\subset I_0$. In particular, for every $x \in I_0$,
\begin{equation}\label{e:ptwRH}
M_{I_0,r}^{\downarrow}(w\chi_{(a,b)})(x)\leq 2M_{I_0}^{\downarrow}(w\chi_{(a,b)})(x) \leq 2[w]_{A_1^{\downarrow}(I_0)} w(x).
\end{equation}

\end{lemma}
\begin{proof}
    Let $I=(a,b) \subset I_0$ and $\lambda \ge \lambda_0 = M^{\downarrow}_{I_0}(\chi_I w)(b)$. Write $\Omega_\lambda = \{ x \in I : M_{I_0}^{\downarrow}(w) > \lambda\}$ as a union of disjoint open intervals $I_j$. For such $I_j$ (cf. \cite{sawyer-one-sided}*{Lemma 2.1}), we have
        \[ \avg{w}_{I_j} = \lambda. \]
    Therefore,
\begin{equation}\label{e:level} w(x \in I : w(x) > \lambda) \le w(\Omega_\lambda) = \lambda |\Omega_\lambda| \le \lambda |x \in I : w(x) \ge \beta \lambda|, \quad \beta = [w]_{A_1^{\uparrow}(I_0)}^{-1}.\end{equation}
For $1<r<2$, we perform the following decomposition
        \[ \int_I w^r = (r-1)\int_0^\infty \lambda^{r-2}w(x \in I : w(x) > \lambda) \, d\lambda, \]
    and consider the integration regions where $\lambda < \lambda_0$ and $\lambda \ge \lambda_0$. The integral over the first region is controlled by the left-hand side in \eqref{e:wRH}. For the second region, we apply \eqref{e:level} to obtain
        \[ \int_{\lambda_0}^\infty \lambda^{r-2}w(x \in I : w(x) > \lambda) \, d\lambda \le r\int_0^\infty \lambda^{r-1} |x \in I : w(x) > \beta \lambda| \, d \lambda = \beta^{-1} \int_I w^r.\]
    Provided that $\frac{r-1}{\beta}$ is smaller than a fixed absolute constant, \eqref{e:wRH} follows. This requirement is precisely the condition imposed on the choice of $r$, which proves the result.
\end{proof}

\begin{lemma}\label{l:Mr-_bound} Let $T$ be an upward-mapping CZO. There exists $C>0$ such that for all $1<p<\infty$, $1<r<2$, and weights $w$ on $I_0$,
\begin{equation*}
\left\|Tf\right\|_{L^{p}(w,I_0)}\leq C \,p p'(r')^{\frac{1}{p'}}\left\|f\right\|_{L^{p}(M_{I_0,r}^{\downarrow}w,I_0)}.
\end{equation*} 
\end{lemma}
Lemma \ref{l:Mr-_bound} directly follows from the case $I_0=\mathbb R$ in \cite{riveros-vidal}*{Theorem 4.1} by restricting to $f$ and $w$ which are supported in $I_0$ and using \eqref{e:MI0}.

\subsection{Weighted weak-type estimate for CZOs}
We now present the main weighted weak-type estimate.

\begin{theorem}\label{thm:main}
Let $T$ be an upward-mapping CZO. Then there exists $C>0$ such that for all $w\in A_{1}^{\uparrow}(I_0)$,
\begin{equation*}
\left\|T (\chi_{I_0} f)\right\|_{L^{1,\infty}(w,I_0)}\leq C[w]_{A_{1}^{\uparrow}(I_0)}\log\left(e+[w]_{A_{1}^{\uparrow}(I_0)}\right)\left\|f\right\|_{L^{1}(w,I_0)}.
\end{equation*}
\end{theorem}

\begin{proof}
Suppose $f\geq0$ and $f\in L^\infty_c(I_0)$; that is, $f$ is bounded and compactly supported in $I_0$.
Introduce
\[
\Omega=\{x\in I_0:M_{I_0}^{\uparrow}f(x)>\lambda\}=\bigcup_jI_j,
\]
where $I_j=(a_j,b_j)$ are the maximal intervals on which $M_{I_0}^{\uparrow}f>\lambda$, satisfying
\[
\frac{1}{|I_j|}\int_{I_j}f(y)\,dy=\lambda.
\]
Define $I_{j}^{+}=(b_{j},d_{j})$, where
\[
d_{j}=\min\left(b_0,\,b_{j}+2|I_{j}|\right).
\]
Thus $I_j^+=(b_j,b_j+2|I_j|)$ if $b_j+2|I_j|\leq b_0$, and otherwise $I_j^+=(b_j,b_0)$. Put $\widetilde{I}_{j}=I_j\cup I_j^{+}$ and $\tilde{\Omega}=\bigcup_j\widetilde{I}_{j}$. 

Decompose $f$ as $f=g+h$, where
\[
g=f\chi_{I_0\setminus\Omega}+\sum_j\lambda\chi_{I_j},\quad h=\sum_j(f-\lambda)\chi_{I_j}.
\]

Note that $g\leq\lambda$ a.e. on $I_0$, $h$ has mean zero on each $I_j$, and $\sum_j \norm{h_j}_{L^1} \le 2 \norm{f}_{L^1}$. We bound $w(\{x\in I_0:|T_{I_0}f(x)|>\lambda\})$ by splitting it into three pieces:
\[
w(\{x:|T_{I_0}f(x)|>\lambda\})\leq w(\tilde{\Omega})+w(E_1)+w(E_2),
\]
where
\[
E_1=\{x\in I_0\setminus\tilde{\Omega}:|T_{I_0}h(x)|>\lambda/2\},\quad
E_2=\{x\in I_0\setminus\tilde{\Omega}:|T_{I_0}g(x)|>\lambda/2\}.
\]
The estimate with $w(\tilde{\Omega})$ relies on the weak-type estimate for $M^{\uparrow}_{I_0}$ in Proposition \ref{thm:weak-type} and the observation that $w(I_j^+) \le 3|I_j| M^{\downarrow}_{I_0}w(x)$ for every $x \in I_j$. Therefore,
    \[ w(\tilde \Omega) = \sum_j w(I_j) + w(I_J^+) \le w(\Omega) + \sum_j \frac{3}{\lambda} \int_{I_j} f(x) M^{\downarrow}_{I_0} w(x) \, dx \lesssim \frac{[w]_{A^{\uparrow}_1(I_0)}}{\lambda} \norm{f}_{L^1(w)}.\] 
We next show that using the cancellation of $h_j$ and the smoothness of the kernel \eqref{e:smooth}, that $w(E_1)$ obeys the same bound as $w(\tilde\Omega)$ above. Since the function $h_{j}$ is supported in $I_{j}$, for $x \in I_0 \backslash \tilde \Omega$, we only need to estimate $Th_j(x)$ when $x > d_j$. We have the inequality:
    \[ \chi_{I_0 \backslash \tilde \Omega}(x)|Th_j(x)| \le C \chi_{[d_j,b_0)}(x) \int_{I_j} \frac{|y-b_j|^\delta}{|x-b_j|^{1+\delta}} |h_j(y)| \, dy.\]
By splitting into dyadic annuli and recalling the precise form of $h_j$,
    \[ \norm{ Th_j}_{L^1(I_0 \backslash \tilde \Omega ,w)} \lesssim \inf_{I_j} M^{\downarrow}_{I_0}(\chi_{I_0 \backslash \tilde I_j} w) \norm {h_j}_{L^1} \lesssim [w]_{A_1^{\uparrow}(I_0)} \norm{f}_{L^1(I_j,w)}.\]
The claimed bound for $w(E_1)$ follows by summing over $j$ and applying Chebyshev's inequality. Finally, we prove the estimate for $ w(E_2)$, which will rely on Lemma \ref{l:wRH} and Lemma \ref{l:Mr-_bound}. We first use Chebyshev's inequality and  Lemma~\ref{l:Mr-_bound} for the localized $L^{p}$ bound for upward‑mapping operators with $r=1+\dfrac{1}{C[ w]_{A_{1}^{\uparrow}(I_0)}}$ and $p = 1 + \frac{1}{\log(e + [ w]_{A_{1}^{\uparrow}(I_0)})}$, both lying in the interval $(1,2)$. These choices of parameters force the factor $(pp'(r')^{\frac{1}{p'}})^{p}$ from Lemma \ref{l:Mr-_bound} to be no larger than $\log \left ( e + [w]_{A_1^{\uparrow}(I_0)} \right)$. Indeed, $\left[(r')^{\frac{1}{p'}}\right]^p$, $p^p$, and $(p')^{p-1}$ are each bounded by an absolute constant, which implies that the main term is $p' \sim \log \left ( e + [w]_{A_1^{\uparrow}(I_0)} \right)$. Furthermore, using that $|g| \le \lambda$, the theorem will be proven if we establish that
    \[ \int_{I_0} |g(x)| M^{\downarrow}_{I_0;r} (w \chi_{I_0 \backslash \tilde \Omega} )(x) \, dx \lesssim [w]_{A_1^{\uparrow}(I_0)}  \norm{f}_{L^1(w)}.\] 
Since $g=f$ outside $\Omega$, from Lemma \ref{l:wRH} we obtain
    \[ \int_{I_0 \backslash \Omega} |g(x)| M^{\downarrow}_{I_0;r} (w \chi_{I_0 \backslash \tilde \Omega} )(x) \, dx \le \int_{I_0 \backslash \Omega} |f(x)| M^{\downarrow}_{I_0;r} (w)(x) \, dx \lesssim [w]_{A_1^{\uparrow}(I_0)}  \norm{f}_{L^1(w)}.\]  
Inside each $I_j$, $g=\lambda=\avg{f}_{I_j}$. Consequently, for any $x,y \in I_j$, 
$$M^{\downarrow}_{I_0;r}(w \chi_{I_0 \backslash \tilde \Omega} w)(x) \lesssim M^{\downarrow}_{I_0;r}(w \chi_{I_0 \backslash \tilde \Omega} w)(y).$$ 
Then, we proceed as in the previously displayed formula. 
The bounds for $w(\tilde{\Omega})$, $w(E_1)$, and $w(E_2)$ yield the desired weak-type $(1,1)$ inequality.
\end{proof}

\subsection{Proof of Proposition \ref{prop:weak}}\label{ss:extrap}
We use a one-sided version of the standard Rubio de Francia algorithm. Given a weight $w \in \Atwoup$, and $\mathbb H$ its associated lower halfspace on which it is non-vanishing, we define $Sh = w^{-1} M^{\downarrow}(|h|w)$. Relying on \cite{sawyer-one-sided}, 
    \[ \norm{S}_{L^2(w)} = \norm{M^{\downarrow}}_{L^2(w^{-1})} \le C [w^{-1}]_{\Atwodown} = C [w]_{\Atwoup}.\]
Furthermore, defining the Rubio de Francia operator
    \[ Rh = \sum_{k=0}^\infty \frac{S^kh}{(2C[w]_{\Atwoup})^k},\]
we obtain that $\norm{R}_{L^2(w))} \le 2$ and $w \cdot Rh \in A_1^{\uparrow}(\mathbb H)$ with $[w \cdot Rh]_{A_1^{\uparrow}(\mathbb H)} \lesssim [w]_{\Atwoup}$. Therefore, given $f,h \in L^2(w)$ of unit norm, and $\lambda>0$, and letting $\Omega = \{ x \in \mathbb H : Tf > \lambda\}$, by Theorem \ref{thm:main} we have
    \[ \int_\Omega h w \le \int_\Omega w \cdot Rh \lesssim [w \cdot Rh]_{A_1^{\uparrow}(\mathbb H)}\log \left( e + [w \cdot Rh]_{A_1^{\uparrow}(\mathbb H)} \right) \norm{f}_{L^1(w \cdot Rh)},\]
But as stated above, $[w \cdot Rh]_{A_1^{\uparrow}(\mathbb H)} \lesssim [w]_{\Atwoup}$. Accordingly, Cauchy-Schwarz shows
    \[ \norm{f}_{L^1(w \cdot Rh)} \lesssim \norm{f}_{L^2(w)} \norm{Rh}_{L^2(w)} \lesssim 1,\]
and taking the supremum over all unit-norm $f$ and $h$ proves the result.

\subsection{Weak-type estimate for the localized one-sided maximal function}\label{ss:max}
We prove a slightly more general statement which implies Proposition \ref{thm:weak-type} as a consequence. We say pair a weights $(u,v)$ on $I_0$ satisfies the localized $A_1^{\uparrow}$ condition on $I_0$ if there exists a constant $C>0$ such that
\begin{equation}\label{e:A1-two-weight}
M_{I_0}^{\uparrow}u(x)\leq Cv(x)\quad\text{for a.e. }x\in I_0.
\end{equation}
\begin{proposition}
\label{p:two-weight-max}
$M_{I_0}^{\downarrow}$ maps $L^1(I_0,\,v\,dx)$ into $L^{1,\infty}(I_0,\,u\,dx)$ if and only if the pair $(u,v)$ satisfies the localized $A_1^{\downarrow}$ condition on $I_0$. Moreover, the bound in Proposition \ref{thm:weak-type} holds.
    
\end{proposition} 

The main difficulty in proving this result is that the one-sided maximal function only satisfies $w(I_j^-) \le M^{\uparrow}_{I_0} w(x)$ for every $x$ in $I_j^+$, while the two-sided maximal function satisfies $w(I) \le Mw(x)$ for every $x$ in $I$. We outline how this obstacle is overcome in \cites{martin-reyes-weighted,sawyer-one-sided}\footnote{These references take the opposite orientation ($\uparrow$ versus $\downarrow$) to ours, but a simple reflection argument connects the two.}. We show that if 
    \[ \{x\in I_0:M_{I_0}^{\downarrow}f(x)>\lambda\}.
=\bigcup_{j}I_j,
\]
where $I_j$ are disjoint, then for any weight $u$,
    \[ u(I_j) \le \frac{1}{\lambda} \int_{I_j} f(x) M^{\downarrow}_{I_0}u(x) \, dx.\]
Then, Proposition \ref{p:two-weight-max} follows by summing over $j$ and invoking the localized $A_1^{\uparrow}$ condition \eqref{e:A1-two-weight}. Let $j$ be fixed and write $I=(a,b)=I_j$. Introduce the auxiliary weight 
    \[ w(x) = \min_{a \le y \le x} M^{\downarrow}(\chi_I u)(y).\] 
Immediately, $w$ is decreasing. Furthermore, from \cite{martin-reyes-weighted}*{Lemma~2(ii)}, we have $u(I) \le w(I)$ for any interval $I$. Thus, if we can establish 
\[
w(I) \le \lambda^{-1} \int_I f(x)\, w(x)\, dx,
\]
the desired estimate follows. We rely on \cite{sawyer-one-sided}*{Lemma~1}, which states that since $M^{\uparrow}(\chi_{I_0} f)(x) > \lambda$ for all $x \in (a,b)$, but not for $x$ in a small neighborhood outside of $I$, we have
\begin{equation}
\label{e:maximal-int}
\frac{1}{t-a} \int_a^t f(x)\, dx \ge \lambda, \qquad \text{for every } t \in (a,b).
\end{equation}
Finally, \cite{martin-reyes-weighted}*{Lemma~1} states that if \eqref{e:maximal-int} holds, then for any decreasing function $w$,
\[
\int_{a}^{b} f(x)\, w(x)\, dx \ge \lambda \int_{a}^{b} w(x)\, dx.
\]
Since our auxiliary weight $w$ is decreasing, Proposition~\ref{p:two-weight-max} follows.

\bibliographystyle{amsplain}

\end{document}